\documentclass[12pt]{amsart}
\usepackage[latin1]{inputenc}
\usepackage{amsxtra,amssymb,amsthm,amsmath,amscd,mathrsfs}
\usepackage[all]{xy}
\def \N {{\mathbb N}}

\def \d {\,{\rm d}}
\def\re{{\Re e\,}}

\def \dm {{\hbox {$\frac{1}{2}$}}}

\def\le{\leqslant}
\def\ge{\geqslant}

\theoremstyle{plain}
\newtheorem{theorem}{Theorem}

\newtheorem{lemma}{Lemma}[section]
\theoremstyle{remark}

\theoremstyle{definition}

\numberwithin{equation}{section}

\begin{document}

%\noindent{{\it Math. Z.}, to appear}

\vskip 5mm

\title[The number of Hecke eigenvalues of same signs]
{The number of Hecke eigenvalues of same signs} 
\author{Y.-K. Lau \& J. Wu}
\address{Department of Mathematics
\\
The University of Hong Kong
\\
Pokfulam Road
\\
Hong Kong}
\email{yklau@maths.hku.hk}
\address{Institut Elie Cartan Nancy (IECN)
\\
Nancy-Universit\'e CNRS INRIA
\\
Boulevard des Aiguillettes, B.P. 239
\\
54506 Van\-d\oe uvre-l\`es-Nancy
\\
France}
\email{wujie@iecn.u-nancy.fr}
\address{School of Mathematical Sciences
\\
Shandong Normal University
\\
Jinan, Shandong 250100
\\
China}

\date{\today}

\begin{abstract}
We give the best possible lower bounds in order of magnitude 
for the number of positive and negative Hecke eigenvalues.
 This improves upon a recent work of Kohnen, Lau \& Shparlinski. 
Also, we study an analogous problem for short intervals. 
\end{abstract}
\maketitle

\addtocounter{footnote}{1}

\section{Introduction}

\smallskip

Let $k\ge 2$ be an even integer and $N\ge 1$ be squarefree. 
Among all holomorphic cusp forms of weight $k$ for the congruence subgroup $\Gamma_0(N)$, 
there are finitely many of them whose Fourier coefficients in the expansion at the cusp $\infty$, 
\begin{equation*}\label{SFf}
f(z) 
= \sum_{n=1}^\infty \lambda_f(n) n^{(k-1)/2} e^{2\pi inz}
\qquad(\Im m z>0),
\end{equation*}
are the Hecke eigenvalues. Up to scalar multiples, these forms are the only simultaneous eigenfunctions of all Hecke operators. 
We call them the primitive forms, and write ${\rm H}_k^*(N)$ 
for the set of all primitive forms of weight $k$ for $\Gamma_0(N)$.  
One central problem in modular form theory is to study the Hecke eigenvalues $\lambda_f(n)$. 
(We omit the factor $n^{(k-1)/2}$ to avoid its uneven amplifying effect.) 
Classically it is known that the arithmetical function $\lambda_f(n)$ is real multiplicative, 
and verifies Deligne's inequality
\begin{equation}\label{Deligne}
|\lambda_f(n)| \le d(n)
\end{equation}
for all $n\ge 1$, 
where $d(n)$ is the divisor function. 
Furthermore we have
\begin{equation}\label{pN}
\lambda_f(p^\nu)=\lambda_f(p)^\nu
\qquad{\rm and}\qquad
\lambda_f(p)=\varepsilon_f(p)/\sqrt{p}
\end{equation}
for all primes $p\mid N$ and integers $\nu\ge 1$,
where $\varepsilon_f(p)\in \{\pm 1\}$.
(See \cite{De74} and \cite{IK}.)
The distribution of the Hecke eigenvalues $\lambda_f(n)$ is delicate. 
The Lang-Trotter conjecture concerns the frequency of $\lambda_f(p)$ taking a value  
in the admissible range where $p$ runs over primes. 
This conjecture is still open but there are progress made on itself or the pertinent questions, 
for instance, \cite{El}, \cite{Ser}, \cite{MMS}, 
\cite{M}, \cite{BO01}, \cite{CFM}, \cite{KRW07}, etc.
In this regard, various techniques and tools are applied, such as $\ell$-adic representations, 
Chebotarev density theorem, sieve-theoretic arguments, 
Rankin-Selberg $L$-functions and the method of ${\mathscr B}$-free numbers. 
In \cite{KRW07}, Kowalski, Robert \& Wu investigated the nonvanishing problem 
and gave the sharpest upper estimate to-date on the gaps 
between consecutive nonzero Hecke eigenvalues.
Another wide belief is Sato-Tate's conjecture, 
asserting that $\lambda_f(p)$'s are equidistributed on $[-2, 2]$ with respect to the Sato-Tate measure.

In this paper, we are concerned with the Hecke eigenvalues of the same sign. 
Kohnen, Lau \& Shparlinski \cite[Theorem 1]{KLS} proved
\begin{equation}\label{LowerBN+N-KLS}
{\mathscr N}_f^{\pm}(x)
:=\sum_{\substack{n\le x, \, (n, N)=1\\ \lambda_f(n)\gtrless\,0}} 1
\gg_f \frac{x}{(\log x)^{17}}
\end{equation}
for $x\ge x_0(f)$.
\footnote{It is worthy to indicate that they gave explicit values 
for the implied constant in $\gg$ and $x_0(f)$.}
Very recently Wu \cite[Corollary]{Wu08} improved this result 
by reducing the exponent 17 to $1-1/\sqrt{3}$,
as a simple application of his estimates on power sums of Hecke eigenvalues.
The exponent $1-1/\sqrt{3}$ can be improved to $2-16/(3\pi)$ if one assumes Sato-Tate's conjecture.

Our first result is to remove the logarithmic factor by the
${\mathscr B}$-free number method, which is the best possible in order of magnitude.

\begin{theorem}\label{LBNpmx}
Let $f\in {\rm H}_k^*(N)$.
Then there is a constant $x_0$ such that the inequality 
\begin{equation}\label{LowerNpm}
{\mathscr N}_f^{\pm}(x)
\gg_f x
\end{equation}
holds for all $x\ge x_0$.
\end{theorem}
\noindent
Remarks. 1. It is clear from the proof that our method gives the stronger result
$$
\sum_{\substack{n\le x, \, (n, N)=1\\ n\,{\rm squarefree}, \,\lambda_f(n)\gtrless\,0}} 1
\gg_f x
$$
for every $x\ge x_0(f)$.

2. The method is robust and applies to, for example, modular forms of half-integral weight. We return to this problem in another occasion.
 
\smallskip

By coupling (\ref{LowerBN+N-KLS}) with Alkan \& Zaharescu's result in \cite[Theorem 1]{AZ05}, 
it is shown in \cite[Theorem 2]{KLS} (see also \cite[Theorem 3.4]{K07})
that there are absolute constants $\eta<1$ and $A>0$ such that
for any $f\in {\rm H}_k^*(N)$ the inequality
\begin{equation}\label{LowerBN+N-KLSshort}
{\mathscr N}_f^{\pm}(x+x^\eta)
-{\mathscr N}_f^{\pm}(x)
>0
\end{equation}
holds for $x\ge (kN)^A$, but no explicit value of $\eta$ is evaluated. 
Apparently it is interesting and important to know how small $\eta$ can be, 
in order for a better understanding of the local behaviour. 
A direct consequence of (\ref{LowerBN+N-KLSshort}) is 
that $\lambda_f(n)$ has a sign-change
in a short interval $[x,x+x^\eta]$ for all sufficiently large $x$. 
The sign-change problem was explored in \cite{IKS07}, \cite{KLS}, \cite{Wu08} on different aspects.
Here we prove that there are plenty of eigenvalues of the same signs 
in intervals of length about $x^{1/2}$. 
More precisely, we have the following.

\begin{theorem}\label{LBNpmxShort}
Let $f\in {\rm H}_k^*(N)$. 
There is an absolute constant $C>0$ such that for any $\varepsilon>0$ 
and all sufficiently large $x\ge N^2x_0(k)$, we have
\begin{equation}\label{LowerNpmShort}
{\mathscr N}_f^{\pm}(x+C_Nx^{1/2})-{\mathscr N}_f^{\pm}(x)
\gg_{\varepsilon} (Nx)^{1/4-\varepsilon},
\end{equation}
where 
$$C_N
:= CN^{1/2}\Psi(N)^3,
\qquad
\Psi(N):=\sum_{d\mid N}d^{-1/2}\log(2d)
$$
and $x_0(k)$ is a suitably large constant depending on $k$ 
and the implied constant in $\gg_\varepsilon$ depends only on $\varepsilon$. 
\end{theorem}

The result in Theorem~\ref{LBNpmxShort} is uniform in the level $N$, 
and its method of proof is based on Heath-Brown \& Tsang \cite{HBT94}.
The exponent of $\Psi(N)$ in $C_N$ can be easily reduced to any number bigger than $3/2$, which however may not be essential as $\Psi(N)$ is already very small - $\log \Psi(N) = o(\sqrt{\log N})$. The range of $x\ge N^2x_0(k)$ can also be refined to $x\ge N^{1+\varepsilon} k^A$ for some constant $A>0$, but we save our effort.

\medskip

\noindent{\bf Acknowledgement}.
Part of this work was done during the visit of the second author 
at the University of Hong Kong in 2008.
He would like to thank the department of mathematics for hospitality. 
The work was supported in part by a grant from the PROCORE-France/Hong Kong Joint Research Scheme sponsored by the Research Grants Council of Hong Kong 
and the Consulate General of France in Hong Kong (F-HK36/07T).
The authors would also thank the referee for his careful reading and helpful comments.

\vskip 10mm

\section{Proof of Theorem \ref{LBNpmx}}

\smallskip

Let $p'$ be the least prime such that $p'\nmid N$ and $\lambda_f(p')<0$.
\footnote{According to \cite{IKS07}, we have $p'\ll (k^2N)^{29/60}$.}
Introduce the set 
\begin{align*}
{\mathscr B}
& = \{p : \lambda_f(p)=0\}
\cup \{p : p\mid N\}
\cup \{p'\}
\cup \{p^2 : p\nmid p'N\;{\rm and}\;\lambda_f(p)\not=0\}
\\
& = \{b_i\}_{i\ge 1}
\quad(\hbox{with increasing order}).
\end{align*}
By virtue of Serre's estimate \cite[(181)]{Ser}:
$$
|\{p\leq x\,:\,\lambda_f(p)=0\}|\ll_{f,\delta} \frac{x}{(\log x)^{1+\delta}}
$$
for $x\geq 2$ and any $\delta<\frac{1}{2}$, 
we infer that 
$$
\sum_{i\ge 1} 1/b_i<\infty
\qquad{\rm and}\qquad
(b_i,b_j)=1
\quad
(i\not=j).
$$
Let ${\mathscr A} := \{a_i\}_{i\ge 1}$ (with increasing order)
be the sequence of all ${\mathscr B}$-free numbers, i.e. the integers  indivisible by any element in $\mathscr{B}$.
According to \cite{Er66}, ${\mathscr A}$ is of positive density 
\begin{equation}\label{Bfree}
\lim_{x\to \infty}\frac{|{\mathscr A}\cap [1, x]|}{x}
=\prod_{i=1}^{\infty}\bigg(1-\frac{1}{b_i}\bigg)>0.
\end{equation}
 From the definition of ${\mathscr B}$ 
and the multiplicativity of $\lambda_f(n)$,
we have $\lambda_f(a)\not=0$ for all $a\in {\mathscr A}$.
Then we partition
$${\mathscr A}={\mathscr A}^{+}\cup {\mathscr A}^{-},$$
where
$$
{\mathscr A}^{\pm}
:= \big\{a_i\in {\mathscr A} : \lambda_f(a_i)\gtrless\,0\big\}.
$$
Without control on the sizes of ${\mathscr A}^\pm$, 
we construct a set from ${\mathscr A}^+\cup {\mathscr A}^-$ 
such that the sign of $\lambda_f(a)$ is switched on the counterpart.
Consider 
$$
{\mathscr N}^{\pm}
:= {\mathscr A}^{\pm}\cup \{a_ip' : a_i\in {\mathscr A}^{\mp}\}.
$$
Clearly $\lambda_f(a)\gtrless 0$ and $(a, N)=1$ for all $a\in {\mathscr N}^{\pm}$ and
$${\mathscr N}_f^{\pm}(x)
\ge \big|{\mathscr N}^{\pm}\cap [1, x]\big|
\ge \big|{\mathscr A}\cap [1, x/p']\big|$$
for all $x\ge 1$.
The desired result follows with the inequality (\ref{Bfree}).

\vskip 10mm

\section{Proof of Theorem \ref{LBNpmxShort}}

\smallskip

The method of proof is based on the investigation of 
$$ 
S_f^{*}(x):=\sum_{n\le x, \, (n, N)=1}\lambda_f(n).
$$ 
Since the $L$-function associated to $f$ is belonged to the Selberg class and of degree 2, 
we apply the standard complex analysis 
to derive truncated Voronoi formulas for $S_f^{*}(x)$.

\begin{lemma}\label{ij}
Let $f\in {\rm H}_k^*(N)$.
Then for any $A>0$ and $\varepsilon>0$, we have
\begin{equation}\label{VoronoiSf*}
\begin{aligned}
S_f^{*}(x) 
& = \frac{\eta_f}{\pi\sqrt{2}} (Nx)^{1/4}
\sum_{d\mid N} \frac{(-1)^{\omega(d)}\lambda_f(d)}{d^{1/4}}
\sum_{n\le M} \frac{\lambda_f(n)}{n^{3/4}} \cos \left(4\pi\sqrt{\frac{nx}{dN}}-\frac{\pi}{4}\right)
\\
& \quad
+ O\left(N^{1/2}\bigg\{1 + \bigg(\frac{x}{M}\bigg)^{1/2} + \bigg(\frac{N}{x}\bigg)^{1/4}
\bigg\}(Nx)^{\varepsilon}\right)
\end{aligned}
\end{equation}
uniformly for $1\le M\le x^A$ and $x\ge N^{1+\varepsilon}$, 
where $\eta_f = \pm 1$ depends on $f$ and the implied $O$-constant 
depends on $A$, $\varepsilon$ and $k$ only.
The function $\omega(d)$ counts the number of all distinct prime factors of $d$.
\end{lemma}

Remark. 
The case $N=1$ and $A=1$ of (\ref{VoronoiSf*}) is covered in \cite[Theorem 1.1]{Ju} 
with $h=k=1$ therein. Our proof follows closely Section 3.2 of \cite{Iv}, and we first evaluate the case without the constraint $(n,N)=1$: for any $A>0$ and $\varepsilon>0$, we have uniformly in  $1\le M\le x^A$, 
\begin{equation}\label{VoronoiSf}
\begin{aligned}
S_f(x)
& := \sum_{n\le x}\lambda_f(n) 
\\
& \,= \frac{\eta_f(Nx)^{1/4}}{\pi\sqrt{2}} 
\sum_{n\le M} \frac{\lambda_f(n)}{n^{3/4}} \cos \bigg(4\pi\sqrt{\frac{nx}N}-\frac{\pi}4\bigg)  
\\
& \quad
+ O\left(N^{1/2}\bigg\{1 + \bigg(\frac{x}{M}\bigg)^{1/2} + \bigg(\frac{N}{x}\bigg)^{1/4}
\bigg\}(Nx)^{\varepsilon}\right).
\end{aligned}
\end{equation}

\begin{proof} 
As usual, denote by $\mu(N)$ the M\"obius function.
(\ref{VoronoiSf*}) follows from (\ref{VoronoiSf}) because
\begin{eqnarray}\label{new1}
S_f^*(x)
&=&\sum_{d|N}\mu(d) \sum_{n\le x/d} \lambda_f(dn)\nonumber\\
&=& \sum_{d|N}(-1)^{\omega(d)} \lambda_f(d)\sum_{n\le x/d} \lambda_f(n)
\end{eqnarray}
by the multiplicativity of $\lambda_f(n)$ and the first equality in (\ref{pN}). 
Note that $x/d\ge x^{\varepsilon/(1+\varepsilon)}$ when $x\ge N^{1+\varepsilon}$ and $d|N$, we can keep the same range of $M$ for all inner sums over $n$ by selecting a suitable $A$. Inserting (\ref{VoronoiSf}) into (\ref{new1}), the main term of (\ref{VoronoiSf*}) comes up immediately. The effect of summing the $O$-terms over $d|N$ is negligible in light of the second formula in (\ref{pN}), and hence the result. 

\vskip3mm

To prove (\ref{VoronoiSf}), we consider $M\in \N$ without loss of generality. 
As usual write
$$
L(s,f)
:=\sum_{n\ge 1}\lambda_f(n)n^{-s}
\qquad 
(\re s>1).
$$
Let $\kappa:=1+\varepsilon$ and $T>1$ be a parameter, chosen as 
\begin{eqnarray}\label{T}
T^2= \frac{4\pi^2(M+\frac12)x}{N}.
\end{eqnarray}
By the truncated Perron formula 
(see \cite[Corollary II.2.4]{Te} 
with the choice of $\sigma_a=1$, $\alpha=2$ and $B(n)=C_{\varepsilon}n^\varepsilon$), 
we have
\begin{equation}\label{3.1}
S_f(x) 
= \frac1{2\pi i} \int_{\kappa-iT}^{\kappa+iT} L(s,f)\frac{x^s}s\d s 
+ O\bigg(N^{1/2}\bigg\{\bigg(\frac{x}{M}\bigg)^{1/2} + 1\bigg\}(Nx)^{\varepsilon}\bigg).
\end{equation}

We shift the line of integration horizontally to $\re s=-\varepsilon$, the main term gives
\begin{equation}\label{3.3}
\frac1{2\pi i} \int_{\kappa-iT}^{\kappa+iT} L(s,f)\frac{x^s}s\d s
= L(0,f) + \frac1{2\pi i} 
\int_{{\mathscr L}} 
L(s,f)\frac{x^s}s\d s,
\end{equation}
where ${\mathscr L}$ is the contour joining the points $\kappa\pm iT$ and $-\varepsilon\pm iT$.
Using the convexity bound 
$$
L(\sigma+it,f)\ll \big(\sqrt{N}(k+|t|)\big)^{\max\{0, 1-\sigma\}+\varepsilon}
\quad
(-\varepsilon\le \sigma\le \kappa),
$$ 
the integrals over the horizontal segments
and the term $L(0,f)$ can be absorbed in $O\big((NTx)^\varepsilon (N^{1/2}+T^{-1}x)\big)$. 
The $O$-constant depends on $k$ and $\varepsilon$, 
and in the sequel, such a dependence in implied constants will be tacitly allowed. 

To handle the integral over the vertical segment 
${\mathscr L}_{\rm v}:=[-\varepsilon-iT, -\varepsilon+iT]$, 
we invoke the functional equation
$$
\bigg(\frac{\sqrt{N}}{2\pi}\bigg)^{s} \Gamma\bigg(s+\frac{k-1}{2}\bigg)L(s,f)
= i^k\eta_f
\bigg(\frac{\sqrt{N}}{2\pi}\bigg)^{1-s} \Gamma\bigg(1-s+\frac{k-1}{2}\bigg)L(1-s,f) 
$$
where $\eta_f := \mu(N)\lambda_f(N)\sqrt{N}\in\{\pm 1\}$
(see \cite[p.375]{IK} with an obvious change of notation). 
Then we deduce that 
\begin{equation}\label{3.4}
\frac1{2\pi i} \int_{{\mathscr L}_{\rm v}} L(s,f)\frac{x^s}s \d s
= i^k \eta_f \sum_{n\ge 1}\frac{\lambda_f(n)}n  I_{{\mathscr L}_{\rm v}}(nx),
\end{equation}
where   
$$
I_{{\mathscr L}_{\rm v}}(y)
:= \frac1{2\pi i} \int_{{\mathscr L}_{\rm v}} 
\left(\frac{4\pi^2}{N}\right)^{s-1/2} 
\frac{\Gamma(1-s+(k-1)/2)}{\Gamma(s+(k-1)/2)} 
\frac{y^s}s \d s.
$$

The quotient of the two gamma factors is 
$$
|t|^{1-2\sigma}e^{-2i(t\log|t|-t)+i{\rm sgn}(t)\pi (k-1)/2}\{1+O(t^{-1})\}
$$ 
for bounded $\sigma$ and any $|t|\ge 1$,
where the implied constant depends on $\sigma$ and $k$.
Together with the second mean value theorem for integrals (see \cite{Te}, Theorem I.0.3), 
we obtain
\begin{equation}\label{3.5}
\begin{aligned}
I_{{\mathscr L}_{\rm v}}(nx)
& \ll N^{1/2}
\left(\frac{N}{nx}\right)^{\varepsilon}
\bigg(\bigg|\int_1^T t^{2\varepsilon} e^{-ig(t)}\d t\bigg| 
+ T^{2\varepsilon}\bigg)
\\
& \ll N^{1/2}
\left(\frac{NT^2}{nx}\right)^\varepsilon 
\bigg(\bigg|\int_a^b e^{-ig(t)}\d t\bigg|+1\bigg)
\end{aligned}
\end{equation}
for some $1\le a\le b\le T$, 
where $g(t) := t\log\big(Nt^2/(4\pi^2nx)\big)-2t$.
In view of (\ref{T}), we have
$$
g'(t)=-\log(4\pi^2nx/(Nt^2))<0
\qquad{\rm and}\qquad
|g'(t)|\ge |\log (n/(M+\textstyle\frac12))|
$$
for $n\ge M+1$ and $1\le t\le T$. 
Using (\ref{Deligne}) and \cite[Theorem I.6.2]{Te},
we infer that 
\begin{equation}\label{3.6}
\begin{aligned}
\sum_{n>M}\frac{\lambda_f(n)}n  I_{{\mathscr L}_{\rm v}}(nx) 
& \ll  N^{1/2}\left(\frac{NT^2}{x}\right)^\varepsilon 
\sum_{n>M} \frac{d(n)}{n^{1+\varepsilon}} 
\bigg(\left|\log \frac{n}{M+\frac12}\right|^{-1} + 1\bigg)
\\
& \ll  N^{1/2}\left(\frac{NT^2}{x}\right)^\varepsilon 
\bigg\{
\sum_{M<n\le 2M} \frac{d(n) (M+\frac12)}{n^{1+\varepsilon} |n-M-\frac12|} 
+ \frac{1}{M^{\varepsilon/2}}\bigg\}
\\
& \ll  N^{1/2}\left(\frac{NT^2}{\sqrt{M}x}\right)^\varepsilon 
\\\noalign{\vskip 2mm}
& \ll N^{1/2}(Nx)^\varepsilon.
\end{aligned}\end{equation}

For $n\le M$, 
we extend the segment of integration ${\mathscr L}_{\rm v}$
to an infinite line ${\mathscr L}_{\rm v}^*$ 
in order to apply Lemma~1 in \cite{CN}.
Write 
$$
{\mathscr L}_{\rm v}^\pm 
:= [\dm+\varepsilon \pm iT, \dm+\varepsilon \pm i\infty),
\qquad
{\mathscr L}_{\rm h}^\pm
:= [-\varepsilon \pm iT, \dm + \varepsilon \pm iT]
$$ 
and define ${\mathscr L}_{\rm v}^*$ to be
 the positively oriented contour consisting of 
${\mathscr L}_{\rm v}$, ${\mathscr L}_{\rm v}^\pm$ and ${\mathscr L}_{\rm h}^\pm$. 
The contribution over the horizontal segments ${\mathscr L}_{\rm h}^\pm$ is 
\begin{align*}
I_{{\mathscr L}_{\rm h}^\pm}(nx) 
& \ll \int_{-\varepsilon}^{1/2-\varepsilon} 
\bigg(\frac{4\pi^2}{N}\bigg)^{\sigma-1/2} T^{1-2\sigma} \frac{(nx)^\sigma}{T} \d\sigma
\\
& \ll N^{1/2} \int_{-\varepsilon}^{1/2-\varepsilon} 
\bigg(\frac{nx}{NT^2}\bigg)^{\sigma} \d\sigma
\\\noalign{\smallskip}
& \ll N^{1/2}(Nx)^\varepsilon.
\end{align*}
As in (\ref{3.5}), for $n\le M$ we get that
\begin{align*}
I_{{\mathscr L}_{\rm v}^\pm}(nx) 
& \ll  N^{1/2} \bigg(\frac{nx}{N}\bigg)^{1/2+\varepsilon}
\bigg(\int_{T}^{\infty} t^{-1-2\varepsilon} e^{-ig(t)} \d t + \frac{1}{T^{1+2\varepsilon}}\bigg)
\\
& \ll N^{1/2}\bigg(\frac{nx}{NT^2}\bigg)^{1/2+\varepsilon} 
\bigg(\left|\log \frac{M+\frac12}n\right|^{-1} + 1\bigg)
\\
& \ll N^{1/2} \bigg(\left|\log \frac{M+\frac12}n\right|^{-1} + 1\bigg).
\end{align*}
So
\begin{equation}\label{3.8}
\begin{aligned}
\sum_{n\le M} \frac{\lambda_f(n)}n  \big(I_{{\mathscr L}_{\rm v}^\pm}(nx)+I_{{\mathscr L}_{\rm h}^\pm}(nx)\big)
& \ll \sum_{n\le M} \frac{d(n)}{n}
\big(\big|I_{{\mathscr L}_{\rm v}^\pm}(nx)\big|+\big|I_{{\mathscr L}_{\rm h}^\pm}(nx)\big|\big)
\\
& \ll N^{1/2}(Nx)^\varepsilon.
\end{aligned}
\end{equation} 

Now all the poles of the integrand in 
$$
I_{{\mathscr L}_{\rm v}^*}(y) 
= \frac{\sqrt{N}}{2\pi} 
\frac1{2\pi i}\int_{{\mathscr L}_{\rm v}^*}  
\frac{\Gamma(1-s+(k-1)/2)\Gamma(s)}{\Gamma(s+(k-1)/2)\Gamma(1+s)} 
\left(\frac{4\pi^2 y}{N}\right)^{s} \d s 
$$
lie on the right of the contour ${\mathscr L}_{\rm v}^*$. 
After a change of variable $s$ into $1-s$, we see that 
\begin{align*} 
I_{{\mathscr L}_{\rm v}^*}(y) 
& = \frac{\sqrt{N}}{2\pi}
I_0\left(\frac{4\pi^2y}{N}\right),
\end{align*}
with 
$$
I_0(t)
:= \frac1{2\pi i}\int_{{\mathscr L}_\varepsilon}  
\frac{\Gamma(s+(k-1)/2)\Gamma(1-s)}{\Gamma(1-s+(k-1)/2)\Gamma(2-s)} 
t^{1-s} \d s. 
$$ 
Here ${\mathscr L}_\varepsilon$ consists of the line $s=\dm-\varepsilon+i\tau$ with $|\tau|\ge T$,
together with three sides of the rectangle whose vertices are 
$\dm-\varepsilon-iT$,
$1+\varepsilon-iT$,
$1+\varepsilon-iT$
and
$\dm-\varepsilon+iT$.
Clearly our $I_0$ is a particular case of $I_\rho$ defined in \cite[Lemma 1]{CN},
corresponding to the choice of parameters $\rho=0$, $\delta=A=1$, $\omega=1$, $h=2$, $k_0=-(2k+1)/4$. 
It hence follows that
\begin{equation}\label{3.9}
I_{{\mathscr L}_{\rm v}^*}(nx)
=  \frac{i^k (nNx)^{1/4}}{\pi\sqrt{2}} 
\cos\left(4\pi\sqrt{\frac{nx}{N}}-\frac{\pi}{4}\right) 
+ O\left(\frac{N^{3/4+\varepsilon}}{(nx)^{1/4}}\right),
\end{equation}
The value of $e_0'$ in Lemma 1 of \cite{CN} is $1/\sqrt{\pi}$ by direct computation.
We conclude 
\begin{equation}\label{3.10}
\begin{aligned}
\sum_{n\le M}\frac{\lambda_f(n)}n  I_{{\mathscr L}_{\rm v}}(nx) 
& = \frac{i^{k}(Nx)^{1/4}}{\pi\sqrt{2}} 
\sum_{n\le M} \frac{\lambda_f(n)}{n^{3/4}} \cos \left(4\pi\sqrt{\frac{nx}{N}}-\frac{\pi}4\right)
\\
& \quad
+ O\left(N^{1/2}\bigg\{\bigg(\frac{N}{x}\bigg)^{1/4} + 1\bigg\}(Nx)^\varepsilon\right),
\end{aligned}
\end{equation}
from  (\ref{3.8}) and (\ref{3.9}), and finally the asymptotic formula (\ref{VoronoiSf}) by
(\ref{3.1})-(\ref{3.4}), (\ref{3.6}) and (\ref{3.10}).
\end{proof}

\smallskip

Following Theorem 1 of \cite{HBT94}, we have the next lemma.

\begin{lemma}\label{HT} 
Let $f\in {\rm H}_k^*(N)$.
There exist positive absolute constants $C, c_1,c_2$ 
such that for all sufficiently large $X\ge N^2X_0(k)$, 
we can find $x_1,x_2\in [X, X+C_NX^{1/2}]$ for which 
$$ 
S_f^*(x_1)> c_1 (NX)^{1/4}
\qquad\mbox{and}\qquad 
S_f^*(x_2)< -c_2 (NX)^{1/4},
$$
where $C_N:= CN^{1/2}\Psi(N)^3$
and $X_0(k)$ is a constant depending only on $k$. 
The same result also holds for $S_f(x)$.
\end{lemma}

\begin{proof} 
Define 
$$
K_\tau(u)
: = (1-|u|)(1+\tau \cos(4\pi\alpha u)),
$$ 
where $\tau =1$ or $-1$ and $\alpha$ is a (large) parameter, both chosen at our disposal. 
Consider the following integral
$$
r_{\beta}
=r_{\beta}(\alpha, \tau, t)
:=\int_{-1}^1 K_\tau(u) \cos\bigg(4\pi(t+\alpha u)\sqrt{\beta}-\frac{\pi}4\bigg)\d u,
$$
where $t\in \N$ and $\beta>0$.
Because
$$
w(\xi)
:=\int_{-1}^1 (1-|u|)e^{i2\pi\xi u}\d u 
= \left(\frac{\sin \pi\xi}{\pi\xi}\right)^2
=\begin{cases}
1                           & \text{if $\xi=0$},
\\
O\big(\min(1,\xi^{-2})\big) & \text{if $\xi\not=0$},
\end{cases}
$$
we can write, with the notation
$\alpha_{\beta}:=2\alpha\sqrt{\beta}$
and
$\alpha_{\beta}^{\pm}:=2\alpha(\sqrt{\beta}\pm 1)$,
\begin{equation}\label{3.11}
\begin{aligned}
r_{\beta}
& =  
\int_{-1}^1 (1-|u|)\bigg(1+\tau\frac{e^{i4\pi\alpha u}+e^{-i4\pi\alpha u}}{2}\bigg)
\re e^{i\{4\pi(t+\alpha u)\sqrt{\beta}-\pi/4\}}\d u
\\
& = \re e^{i(4\pi t\sqrt{\beta}-\pi/4)} \int_{-1}^1 (1-|u|)
\bigg(e^{i2\pi\alpha_\beta u}
+\frac{\tau}{2}e^{i2\pi\alpha_{\beta}^{+}u}
+\frac{\tau}{2}e^{i2\pi\alpha_{\beta}^{-}u}
\bigg)\d u
\\
& = \bigg(w\big(\alpha_{\beta}\big)
+ \frac{\tau}{2}w\big(\alpha_{\beta}^{+}\big) 
+ \frac{\tau}{2}w\big(\alpha_{\beta}^{-}\big) 
\bigg) 
\cos\bigg(4\pi t\sqrt{\beta}-\frac{\pi}{4}\bigg)
\\
& = \delta_{\beta=1}\frac{\tau}{2\sqrt{2}}
+ O\bigg(\min\bigg(1,\frac1{\alpha^2\beta}\bigg)
+\delta_{\beta\not=1}\min\bigg(1,\frac1{(\alpha_{\beta}^{-})^2}\bigg)\bigg),
\end{aligned}
\end{equation}
where the $O$-constant is absolute,
$$
\delta_{\beta=1}:=\begin{cases}
1 & \text{if $\beta=1$}
\\
0 & \text{otherwise}
\end{cases}
\qquad{\rm and}\qquad
\delta_{\beta\not=1}:=1-\delta_{\beta=1}.
$$ 
The last error term in (\ref{3.11}) appears only when $\beta\not=1$.

For all $X\ge N^2X_0(k)$ (whose value will be specified below), 
we write $T=(X/N)^{1/2}$ and $t=[T]+1\in\N$, 
and consider the convolution
$$
J_\tau
=\int_{-1}^1 F_f(t+\alpha u) K_\tau(u)\d u, 
$$
where 
$$
F_f(t+\alpha u)
:= \frac{\pi\sqrt{2}}{\eta_f} \frac{S_f^*(N(t+\alpha u)^2)}{\sqrt{N(t+\alpha u)}}.
$$
By Lemma \ref{ij} with $M= NT^2=X$, we deduce that
\begin{align*}
F_f(t+\alpha u) 
& = \sum_{d\mid N} \frac{(-1)^{\omega(d)}\lambda_f(d)}{d^{1/4}} 
\sum_{n\le M} \! \frac{\lambda_f(n)}{n^{3/4}}
\cos\bigg(4\pi (t+\alpha u)\sqrt{\frac{n}{d}}-\frac{\pi}{4}\bigg)
\! + O_k\bigg(\frac{1}{T^{1/4}}\bigg),
\end{align*}
and 
\begin{align}\label{J}
J_\tau
& = \sum_{d\mid N} \frac{(-1)^{\omega(d)}\lambda_f(d)}{d^{1/4}} 
\sum_{n\le M} \frac{\lambda_f(n)}{n^{3/4}}r_{n/d}+O_k\bigg(\frac{1}{T^{1/4}}\bigg)
\end{align}
by (\ref{pN}).

Next we estimate the contribution of the $O$-term in (\ref{3.11}) to $J_\tau$. 
Using (\ref{pN}) and (\ref{Deligne}) again, its contribution to $J_\tau$ is 
\begin{align}\label{jerr}
& \ll \sum_{d\mid N} \frac{1}{d^{3/4}} \bigg\{\sum_{n\le M} 
\frac{d(n)}{n^{3/4}}R_{d,n}'(\alpha) 
+ \sum_{\substack{n\le M\\ n\neq d}} \frac{d(n)}{n^{3/4}} 
R_{d,n}''(\alpha)\bigg\},
\end{align}
where
$$
R_{d,n}'(\alpha):=\min\bigg(1,\frac{d}{\alpha^2 n}\bigg),
\qquad
R_{d,n}''(\alpha):=\min\bigg(1,\frac{d}{\alpha^2 |\sqrt{n}-\sqrt{d}|^2}\bigg).
$$
Consider the second sum in the curly braces. We separate $n$ into 
$$
n\le\alpha_{-}d, 
\qquad
\alpha_{-}d<n <\alpha_{+}d
\qquad 
{\rm or}
\qquad 
\alpha_{+}d\le n
$$
where $\alpha_\pm :=(1-\alpha^{-1/2})^{\mp 2}$, 
and $R_{d,n}''(\alpha)$ is $\le 1/{\alpha}$, $1$  or $d/(\alpha n)$ accordingly.
Therefore,
$$
\sum_{\substack{n\le M\\ n\neq d}} 
\frac{d(n)}{n^{3/4}} R_{d,n}''(\alpha)
\le \frac1{\alpha} \sum_{n\le \alpha_{-}d} \frac{d(n)}{n^{3/4}}
+ \sum_{\substack{\alpha_{-}d<n<\alpha_{+}d\\ n\neq d}} \frac{d(n)}{n^{3/4}} 
+ \frac{d}{\alpha}\sum_{n>\alpha_{+}d} \frac{d(n)}{n^{7/4}}.
$$
Obviously the first and last terms on the right-hand side are $\ll \alpha^{-1}d^{1/4}\log(2d)$.
Note that $n\asymp d$ in the second sum.
So, by using Shiu's Theorem 2 in \cite{Shiu}
it follows
\begin{align*}
\sum_{\substack{\alpha_{-}d<n<\alpha_{+}d\\ n\neq d}} \frac{d(n)}{n^{3/4}} 
& \ll d^{-3/4}\sum_{\substack{\alpha_{-}d<n<\alpha_{+}d\\ n\neq d}} d(n)
\\
& \ll \alpha^{-1/2}d^{1/4}\log(2d)
\end{align*}
if $d>\alpha$.
Otherwise (i.e. $d\le \alpha$), 
pulling out $d(n)\ll n^\varepsilon\ll d^\varepsilon\ll \alpha^\varepsilon$, we have
\begin{align*}
\sum_{\substack{\alpha_{-}d<n<\alpha_{+}d\\ n\neq d}} d(n)n^{-3/4}
& \ll \alpha^\varepsilon d^{-3/4} \sum_{\substack{\alpha_{-}d<n<\alpha_{+}d\\ n\neq d}} 1 
\\\noalign{\vskip -1mm}
& \ll \alpha^\varepsilon d^{-3/4} \alpha^{-1/2} d
\\\noalign{\vskip 4mm}
& \ll \alpha^{-1/3} d^{1/4}\log(2d).
\end{align*}
(We can assume that $(\alpha_{+}-\alpha_{-})d\ge \alpha^{-1/2}d\ge c'$ 
for a small constant $c'$,
otherwise the last sum is empty.)
Hence
$$
\sum_{\substack{n\le M\\ n\neq d}} 
\frac{d(n)}{n^{3/4}} R_{d,n}''(\alpha)
\ll \alpha^{-1/3} d^{1/4}\log(2d).
$$

The first sum in the bracket of (\ref{jerr}) can be treated in the same fashion (even more easily). 
Thus, (\ref{jerr}) is bound by 
$$
\ll \alpha^{-1/3} \sum_{d\mid N}\frac{\log(2d)}{d^{1/2}} 
=: \alpha^{-1/3} \Psi(N).
$$

We conclude from (\ref{J}) with (\ref{3.11}) and (\ref{pN}) that 
\begin{align*}
J_\tau = \frac{\tau}{2\sqrt{2}} \sum_{d\mid N} \frac{(-1)^{\omega(d)}}{d^2}
+O\bigg(\frac{\Psi(N)}{\alpha^{1/3}}\bigg)+O_k\bigg(\frac{1}{T^{1/4}}\bigg),
\end{align*}
where the implied constant is absolute in the first $O$-term, but depends on $k$ in the second. 
Noticing that
$$
\sum_{d\mid N} \frac{(-1)^{\omega(d)}}{d^2}
=\prod_{p\mid N}\bigg(1-\frac{1}{p^2}\bigg)
\ge \frac{6}{\pi^2}
$$
and $T\ge \sqrt{NX_0(k)}$, 
we take $\alpha=C\Psi(N)^3$ with a large absolute constant $C$ and 
a large $X_0(k)$ so that both $O$-terms $O(\alpha^{-1/3}\Psi(N))$ and $O_k(T^{-1/4})$ 
are $\le \cos (\pi/4)/\pi^2=1/(\pi^2\sqrt{2})$. Therefore
$$
J_{-1}<-1/(\pi^2\sqrt{2})
\qquad{\rm and}\qquad
J_1>1/(\pi^2\sqrt{2}).
$$ 

With the nonnegativity of $K_\tau(u)$ and the estimate
$$
1-(2\pi\alpha)^{-2}\le \int_{-1}^1 K_\tau (u)\d u \le 2\qquad (\tau = \pm 1), 
$$
we have 
$$2F_f(t+\alpha \eta_{+})\ge 1/(\pi^2\sqrt{2})
\quad{\rm and}\quad 
\big(1-(2\pi\alpha)^{-2}\big)F_f(t+\alpha \eta_{-})\le -1/(\pi^2\sqrt{2})
$$
for some $\eta_{+},\eta_{-}\in [-1,1]$. 
Let $C_N= CN^{1/2}\Psi(N)^3$. As 
$$
X-3C_N\sqrt{X}\le N(t+\alpha\eta_{\pm})^2 \le X+3C_N\sqrt{X},
$$ 
our assertion follows from the definition of $F_f$ and replacing $X-3C_N\sqrt{X}$ by $X$.
\end{proof} 

\medskip

Now we are ready to prove Theorem \ref{LBNpmxShort}.

We exploit the consecutive sign changes of $S_f^*(x)$. 
Let $x\ge N^2X_0(k)$ where $X_0(k)$ takes the value as in Lemma~\ref{HT}.
We apply Lemma \ref{HT} to the intervals 
$[x,x+C_Nx^{1/2}]$ and $[y, y+C_Ny^{1/2}]$ 
where $y=x+C_Nx^{1/2}$. 
Over each of the intervals, $S_f^*(x)$ attains in magnitude $(Nx)^{1/4}$ 
in both positive and negative directions. 
Hence, we can find three points $x<x_1< x_2<x_3 <x+3C_Nx^{1/2}$ 
such that $S_f^*(x_i)$ $(i=1,2,3)$ takes alternate signs and their absolute values are $\gg (Nx)^{1/4}$. 
(Note that $2\sqrt{x} \ge \sqrt{x+C_N\sqrt{x}}$.) 
It follows that the two differences
$$
S_f^*(x_2)- S_f^*(x_1)
=\sum_{\substack{x_1< n\le x_2\\ (n, N)=1}} \lambda_f(n)
$$
and
$$
S_f^*(x_3)- S_f^*(x_2)
=\sum_{\substack{x_2< n\le x_3\\ (n, N)=1}} \lambda_f(n)
$$
have absolute values $\gg (Nx)^{1/4}$ but are of opposite signs. 
This implies (\ref{LowerNpmShort}), 
since for example, if 
$$
\sum_{\substack{a<n<b\\ (n, N)=1}}\lambda_f(n)< -c'(Nx)^{1/4}
$$ 
for some constant $c'>0$ and $b\ll x$, 
then we have
\begin{align*}
c'(Nx)^{1/4} 
& <\sum_{\substack{a<n<b, \,(n, N)=1\\ \lambda_f(n)<0}}\big(-\lambda_f(n)\big)
\\
& \ll x^\varepsilon \sum_{\substack{a<n<b, \,(n, N)=1\\ \lambda_f(n)<0}} 1.
\end{align*}
This completes the proof of Theorem \ref{LBNpmxShort}.
\hfill
$\square$

\end{document}